\documentclass[11pt,oneside,a4paper]{article}
\usepackage{authblk}
\usepackage{graphicx}
\usepackage[intlimits]{amsmath}
\usepackage{amsthm}
\usepackage{amsmath}
\usepackage{amssymb}
\usepackage[T1]{fontenc}
\usepackage{color}
\usepackage{mathrsfs}
\usepackage{hyperref}
\usepackage{bbm}

\usepackage{array}
\usepackage{makecell}
\usepackage[normalem]{ulem}
\usepackage{enumitem}
\usepackage{fancyhdr}

\newtheorem{thm}{Theorem}
\newtheorem{lem}{Lemma}
\newtheorem{rem}[lem]{Remark}

\newtheorem{cor}{Corollary}

\newtheorem{prop}[lem]{Proposition}

\def \pK{\mathcal{K}}
\def \NN{\mathbb{N}}

\def \RR{\mathbb{R}}
\def \Rd{\RR^d}

\def \mJ{p}
\def \dJ{r}

\DeclareMathOperator {\supp}{supp}
\definecolor{tjb}{rgb}{1,0.0,1}
\newcommand{\ind}{{\bf 1}}

\definecolor{ksb}{rgb}{0.7,0.1,0.2}

\AtEndDocument{%
  \par
  \medskip
  \begin{tabular}{@{}l@{}}%
    \textsc{Tomasz Jakubowski and Karol Szczypkowski}\\[2mm]
		Wroclaw University of Science and Technology\\
		27 Wybrze\.ze Wyspia\'nskiego st., 50-370 Wroc\l{}aw, Poland \\[2mm]
    \textit{E-mail addresses}: \\
		\texttt{tomasz.jakubowski@pwr.edu.pl},
		\texttt{karol.szczypkowski@pwr.edu.pl}
  \end{tabular}}

\title{Sharp and plain estimates
for Schr{\"o}dinger perturbation \\ of Gaussian kernel}

\author{Tomasz Jakubowski and Karol Szczypkowski}
\date{\empty}

\begin{document}

\maketitle

\begin{abstract}
We investigate whether
a fundamental solution of the Schr{\"o}dinger equation $\partial_t u =(\Delta +V)\, u$
has local in time sharp Gaussian estimates.
We compare 
that class with the class of $V$
for which
local in time plain Gaussian estimates hold. 
We concentrate on 
$V$ that have fixed sign and
we present certain conclusions for $V$ in the Kato class.
\end{abstract}

\section{Introduction and main results}
\makeatletter{\renewcommand*{\@makefnmark}{}
\footnotetext{AMS subject classification: Primary 47D06, 47D08; Secondary 35A08, 35B25}
\footnotetext{Keywords and phrases: Schr\"odinger perturbation, Gaussian estimates}
\footnotetext{The research of the first author was supported by the NCN grant 2015/18/E/ST1/00239\makeatother}}

Let $d=1,2,\ldots$.
We consider the Gauss-Weierstrass kernel,
\[g(t,x,y)=
(4\pi t)^{-d/2} e^{-\frac{|y-x|^2}{4t}}, \qquad t>0,\ x,y\in\Rd.\]
It is well known that $g$ is the fundamental solution of the equation $\partial_t u=\Delta u$, and time-homogeneous probability transition 
density -- the heat kernel of $\Delta$.
Throughout the paper we let
$V: \Rd\to \RR$
to be a Borel measurable function.
We call 
$G:(0,\infty)\times \Rd\times \Rd\to [0,\infty]$
the heat kernel of $\Delta+V$
or the Schr\"odinger perturbation of $g$ by $V$, if the following
Duhamel or perturbation formula holds for $t>0$, $x,y\in \Rd$,
\[
G(t,x,y)=g(t,x,y)+\int_0^t \int_{\Rd} G(s,x,z)V(z)g(t-s,z,y)dzds.
\]
One of the directions in the study of $G(t,x,y)$
is to find its estimates or bounds.
It is  
natural to 
ask if there are positive numbers, i.e., {\it constants} $0<c_1\le c_2<\infty$ such that the following two-sided bound holds,
\begin{align}\label{est:sharp_uni}
c_1  \leq \frac{G(t,x,y)}{g(t,x,y)}\leq c_2,\qquad t>0,\ x,y\in \Rd.
\end{align}
We call \eqref{est:sharp_uni}
{\it sharp Gaussian estimates} (or bounds) {\it global} (or uniform) in time.
One can also ponder a weaker property -- if for a given $T\in (0,\infty)$,
\begin{align}\label{est:sharp_time}
c_1  \leq \frac{G(t,x,y)}{g(t,x,y)}\leq c_2 \,,\qquad 0<t\le T,\ x,y\in \Rd.
\end{align}
We call \eqref{est:sharp_time} {\it sharp Gaussian estimates local} in time.
We observe that the inequality in \eqref{est:sharp_uni} is stronger than
the 
{\it
plain 
Gaussian estimates global} in time
\begin{align}\label{est:gaus} 
c_1\, (4\pi t)^{-d/2} e^{-\frac{|y-x|^2}{4t\varepsilon_1}} \leq G(t,x,y)\leq c_2\, (4\pi t)^{-d/2} e^{-\frac{|y-x|^2}{4t\varepsilon_2}},\qquad t>0,\ x,y\in \Rd,
\end{align}
where
$0<\varepsilon_1 \le 1\le \varepsilon_2<\infty$.
Similarly, \eqref{est:sharp_time} is stronger than
the 
{\it
plain 
Gaussian estimates local} in time
\begin{align}\label{est:gaus_b} 
c_1\, (4\pi t)^{-d/2} e^{-\frac{|y-x|^2}{4t\varepsilon_1}} \leq G(t,x,y)\leq c_2\, (4\pi t)^{-d/2} e^{-\frac{|y-x|^2}{4t\varepsilon_2}},\qquad 0<t\leq T,\ x,y\in \Rd.
\end{align}
We refer the reader to
\cite{MR3914946}
and \cite{MR3200161}
for a brief survey on the literature
concerning \eqref{est:sharp_uni}, \eqref{est:sharp_time}, 
\eqref{est:gaus}
and \eqref{est:gaus_b}, in particular,
on the results of
\cite{MR1978999}, 
\cite{MR1994762} and \cite{MR4093916}.
In the present paper our main focus is on
the distinction between local  sharp Gaussian estimates \eqref{est:sharp_time}
and local plain Gaussian estimates
\eqref{est:gaus_b}.

In Theorem~\ref{thm:t1}
we combine our findings with those of \cite{MR3914946} to depict when for $V\leq 0$
local (or global) {\it sharp} Gaussian estimates 
hold
if and only if local (or global) {\it plain} Gaussian estimates hold. 
\begin{thm}\label{thm:t1}
Let $V\leq 0$. Then, \eqref{est:sharp_time} holds if and only if
 \eqref{est:gaus_b} holds
according to the {\rm 'local in time'} column of Table~\ref{t:1}.
Similarly,
 \eqref{est:sharp_uni} holds if and only if \eqref{est:gaus} holds according to the
{\rm 'global in time'} column.
\begin{table}[h!]\label{t:1}
\begin{center}
\begin{tabular}[t]{ m{3cm} |c|l|l| m{3cm}}
\cline{2-4}
&dimension & \makecell{local\\in time} & \makecell{global\\ in time} & \\ %[0.5ex]
%\hline\hline
\cline{2-4}
\noalign{\vskip\doublerulesep
         \vskip-\arrayrulewidth}
\cline{2-4} 
& $d\geq 4$ & No & No   \\  
\cline{2-4}
& $d=3$ & No\textsuperscript{\tiny{1)}} & Yes\textsuperscript{\tiny{2)}}   \\  
\cline{2-4}
& $d=2$ & Yes\textsuperscript{\tiny{3)}} & Yes\textsuperscript{\tiny{5)}}  \\  
\cline{2-4}
& $d=1$ & Yes\textsuperscript{\tiny{4)}} & Yes\textsuperscript{\tiny{5)}}  \\  
\cline{2-4}
\end{tabular}
\caption{
Equivalence of {\it sharp} and {\it plain} Gaussian bounds for $V\leq 0$.}
\end{center}
\end{table}
\end{thm}
At this point we enclose some comments and references that complete Table~\ref{t:1} and which can also be tracked in other places in the paper.
\begin{rem}
Let $V\leq 0$. We list the superscripts of Table~\ref{t:1}.
\begin{enumerate}
\item[{\rm 1)}] we refer the reader to \cite[Theorem~1B]{MR1994762};
\item[{\rm 2)}] \eqref{est:sharp_uni} and \eqref{est:gaus} are equivalent to the potential boundedness of $V$ if $d=3$, see \cite{MR3914946};
\item[{\rm 3)}] \eqref{est:sharp_time} and \eqref{est:gaus_b} are equivalent to the enlarged Kato class condition on $V$ if $d=2$, see \eqref{enKato} and Corollary~\ref{cor:d2_1};
\item[{\rm 4)}] \eqref{est:sharp_time} and \eqref{est:gaus_b} are equivalent to Kato class condition on $V$ (uniform local integrability of $V$) if $d=1$, see \eqref{Kato} and Corollary~\ref{cor:d1_1};
\item[{\rm 5)}] \eqref{est:sharp_uni} as well as \eqref{est:gaus} are impossible for non-trivial $V$ if $d\leq 2$, see \cite[page 3]{MR3914946}.
\end{enumerate}
\end{rem}

In the literature
there exist several intrinsic quantities
that are used to
characterize $V \leq 0$
for which \eqref{est:sharp_time} holds,
and
to formulate 
necessary and (separately) sufficient conditions
for \eqref{est:sharp_time} if $V\geq 0$.
Let us start with one that
derives from
Zhang \cite[Lemma~3.1 and Lemma~3.2]{MR1978999} and from Bogdan, Jakubowski and Hansen  \cite[(1)]{MR2457489}.
For $t>0$ and $x,y\in\Rd$
we define
\begin{align}\label{def:S}
S(V,t,x,y)=\int_0^t \int_{\Rd} \frac{g(s,x,z)g(t-s,z,y)}{g(t,x,y)}|V(z)|\,dzds\,.
\end{align}
Further, we let
\begin{align*}
\|S(V,t)\|_{\infty}= \sup_{x,y\in\Rd} S(V,t,x,y)\,,\qquad
\|S(V)\|_{T,\infty}=\sup_{0<t\leq T} \|S(V,t)\|_{\infty}\,.
\end{align*}
Other quantities are surveyed in Section~\ref{sec:overview}.
The following lemma is an excerpt from 
\cite{MR3914946}
that 
exposes the relation between $\|S(V)\|_{T,\infty}$
and \eqref{est:sharp_time},
and 
will
suffice for our discussion and purposes.
We write as usually $f^+=\max\{0,f\}$, $f^-=\max\{0,-f\}$.

\begin{lem}\label{lem:comb}
We have
\begin{enumerate}
\item[1)] If $V\leq 0$, then
for each $T\in (0,\infty)$,
\eqref{est:sharp_time} is equivalent to
$\|S(V)\|_{T,\infty}<\infty$.
\item[2)] If $V \geq 0$, then \eqref{est:sharp_time} implies
$\|S(V)\|_{T,\infty}<\infty$
for each $T\in (0,\infty)$.  
\item[3)] If for some
$h>0$ and $0\le \eta<1$ we have
$\|S(V^+)\|_{h,\infty}\leq \eta$
and if $S(V^-,t,x,y)$ is bounded on bounded subsets of
$(0,\infty)\times\Rd\times\Rd$,
then
\begin{align}\label{gen_est}
e^{-S(V^-,t,x,y)} \leq \frac{G(t,x,y)}{g(t,x,y)}\leq \left(\frac{1}{1-\eta}\right)^{1+t/h}, \qquad t>0, \ x,y\in \Rd \,.
\end{align}
\end{enumerate}
\end{lem}

The relation between the bound of \eqref{def:S} and the upper bound in \eqref{gen_est}, in the framework of integral kernels, can be found in \cite{MR3000465}.  For some other variants see \cite{MR4058740}.
Recall that the celebrated sufficient condition 
for the local plain Gaussian estimates \eqref{est:gaus_b}  is that $V$ belongs to {\it the Kato class} 
(\cite{MR644024}, \cite{MR670130},
\cite{MR333833}, \cite{MR0203473}),
which we abbreviate to 
$V\in \pK_d$. More precisely, $V\in \pK_d$ if 
\begin{align}\label{Kato}
\lim_{t\to 0^+} \sup_{x\in\Rd} \int_0^t\int_{\Rd}  g(s,x,z) |V(z)|\,dzds=0\,.
\end{align}
We say that
$V$ belongs to 
{\it the enlarged Kato class},
which we denote by $V\in \widehat{\mathcal{K}}_d$,
 if 
\begin{align}\label{enKato}
\sup_{x\in\Rd} \int_0^t\int_{\Rd}  g(s,x,z) |V(z)|\,dzds<\infty\,,
\end{align}
holds for some (every) $t>0$
(see \cite[Proposition~5.1]{MR845197}).
The class $\widehat{\mathcal{K}}_d$ is also known as {\it the Dynkin class} in a measure theory context.
We refer the reader to
\cite{MR1132313},
\cite{MR2345907}
and
\cite{MR3713578}
for a wider perspective on the Kato class;
and to 
\cite{MR1488344},  
\cite{MR1687500},
\cite{MR1783642},
\cite{MR1994762},
\cite{MR2253015},
\cite{MR2253111},
\cite{MR3914946}
for a corresponding class and results for time-dependent $V$.
We will also use the following notation
$$
\Delta^{-1} V (x)= - \int_0^{\infty}\int_{\Rd} g(s,x,z) V(z)\,dzds\,,
\qquad
\qquad
\|\Delta^{-1}V\|_{\infty}=\sup_{x\in \Rd} |\Delta^{-1} V (x)|\,.
$$

We give main results concerning the difference between sharp and plain Gaussian estimates.
We distinguish four cases: $d\geq 4$, $d=3$, $d=2$, $d=1$.
\begin{thm}\label{thm:dgeq4}
Let $d\geq 4$. There exists $V \leq 0$ with the following properties
\begin{enumerate}
\item[{\rm(a)}] ${\rm supp} (V)\subseteq B(0,1)$,
\item[{\rm(b)}] $V\in\pK_d$,
\item[{\rm(c)}] $\|\Delta^{-1} V\|_{\infty}<\infty$,  
\item[{\rm(d)}] $\|S(V,t)\|_{\infty}=\infty$ for every $t>0$.
\end{enumerate}
\end{thm}

Such a strong result is not possible if $d=3$. Indeed, in this dimension the condition 
$\|\Delta^{-1} V\|_{\infty}<\infty$ implies (is equivalent to) $\sup_{t>0}\|S(V,t)\|_{\infty}<\infty$, see
\cite[(7) and (8)]{MR3914946}. In particular, if $V\in\pK_d$ has compact support,
then $\|\Delta^{-1} V\|_{\infty}<\infty$.

\begin{thm}\label{thm:d3}
Let $d=3$. There exists $V\leq 0$ with the following properties
\begin{enumerate}
\item[{\rm(a)}] $V\in\pK_3$,
\item[{\rm(b)}] $\|S(V ,t)\|_{\infty}=\infty$ for every $t>0$.
\end{enumerate}
\end{thm}
Theorems \ref{thm:dgeq4} and \ref{thm:d3} yield that for $d \ge 3$ there is a function $V\leq 0$
such that \eqref{est:gaus_b} holds with $\varepsilon_1< 1<\varepsilon_2$ arbitrarily close to $1$ and \eqref{est:sharp_time}
does not hold. Additionally, for $d \ge 4$ the function $V$ may be chosen in a such a way that $\supp V$ is compact and \eqref{est:gaus} 
holds, see Corollaries~\ref{cor:d4} and~\ref{cor:d3}.
We note that the latter cannot be done in the dimension $3$. In fact, if $d=3$ and $V\leq 0$, the global plain Gaussian estimates \eqref{est:gaus}  hold if and only if global sharp Gaussian estimates \eqref{est:sharp_uni} hold, see \cite[Page 6]{MR3914946}.
From Theorem \ref{thm:d3} we deduce that such phenomenon does not occur for local in time bounds.

The situation is radically different if $d\le2$. In this case the condition $V \in \mathcal{K}_d$ yields $\|S(V ,t)\|_{\infty}<\infty$. It is a consequence of the following theorem.

\begin{thm}\label{thm:d2d1}
Let $d=2$ or $d=1$. There exists an absolute constant $c>0$ such that for all $T>0$ and $V$ we have
\begin{align}\label{ineq:d2d1}
c^{-1} \sup_{x\in\Rd} \int_0^T\int_{\Rd} g(s,x,z) |V(z)|dzds \leq   
\|S(V)\|_{T,\infty} \leq c \sup_{x\in\Rd} \int_0^T \int_{\Rd} g(s,x,z) |V(z)|dzds\,.
\end{align}
\end{thm}
As a corollary of Theorem \ref{thm:d2d1}, we characterize classes $\mathcal{K}_d$ and $\widehat{\mathcal{K}}_d$ for $d\le 2$, by using the quantity $\|S(V)\|_{T,\infty}$, see Corollaries~\ref{cor:d2_1} and~\ref{cor:d1_1}. Additionally, we obtain that for $d\le 2$ and $V\le0$, \eqref{est:sharp_time} holds if and only if $V\in \widehat{\mathcal{K}}_2$. For $d=1$ the same property holds for $V\ge0$. See Corollaries~\ref{cor:d2_1} and~\ref{cor:d1_2}.

The rest of the paper is organized as follows. In Section~\ref{sec:prel} we collect other quantities used in the literature to analyse \eqref{est:sharp_time}, and we show that they are comparable. We also discuss analogies with various descriptions of the Kato class.
In Section~\ref{sec:new_test} we introduce an explicit kernel $K(t,x,y)$ and use it to propose another test for \eqref{est:sharp_time} to hold.
In that section we also formulate and prove Theorem~\ref{thm:new_equiv}. In Section~\ref{sec:proofs}
we prove Theorems \ref{thm:dgeq4} -- \ref{thm:d2d1}.
In Section 5 we give  corollaries of the main results of the paper
and the proof of Theorem~\ref{thm:t1}.

Throughout the paper $B(x,r)$ denotes a ball of radius $r>0$ in $\Rd$ centred at $x\in\Rd$. In short we write $B_r=B(0,r)$.

\subsection*{Acknowledgements}
We thank Krzysztof Bogdan for helpful comments on  the paper.

\section{Preliminaries}\label{sec:prel}

\subsection{An overview of tests for sharp bounds}\label{sec:overview}

We have already seen in Lemma~\ref{lem:comb}
how to use a test based on $S(V,t,x,y)$
to analyse \eqref{est:sharp_time}.
In \cite{MR1978999} Zhang introduced yet another object,
for $t>0$ and $x,y\in\Rd$,
\begin{align*}
N(V,t,x,y)=&\int_0^{t/2}\int_{\Rd}\frac{e^{-|z-y+(\tau/t)(y-x)|^2/(4\tau)}}{\tau^{d/2}}|V(z)|dzd\tau\\  %\label{def:N} \\
&\qquad+\int_{t/2}^t\int_{\Rd} \frac{e^{-|z-y+(\tau/t)(y-x)|^2/(4(t-\tau))}}{(t-\tau)^{d/2}} |V(z)|dzd\tau \,. %\nonumber
\end{align*}
It is actually comparable with $S$ in the following sense,
\begin{align}
S(V,t,x,y)&\geq m_1\, N(V,t/2,x,y)\,,\tag{L} \label{L}\\
S(V,t,x,y)&\leq m_2\, N(V,t,x,y)\,,\tag{U} \label{U}
\end{align}
where constants $m_1$, $m_2$ depend only on $d$, see \cite[(L) and (U) on page 5]{MR3914946}.
The quantity $N$ gives rise to
\begin{align*}
\|N(V,t)\|_{\infty}= \sup_{x,y\in\Rd} N(V,t,x,y)\,,\qquad
\|N(V)\|_{T,\infty}= \sup_{0<t\leq T} \|N(V,t)\|_{\infty}\,.
\end{align*}
On the other hand, in  \cite{MR1994762} Milman and Semenov (for $d\geq 3$) proposed to use for $\lambda> 0$,
$$
e_*(V,\lambda)= \sup_{\alpha\in\Rd} \|(\lambda-\Delta +2\alpha \cdot \nabla)^{-1} |V| \|_{\infty}\,.
$$
The operator $(\lambda-\Delta +2\alpha \cdot \nabla)^{-1}$
is an integral operator with a kernel equal to
$\int_0^{\infty} e^{-\lambda s} \mJ_{\alpha}(s,x,y) ds$,
where
for $\alpha\in\Rd$ and
$t>0$, $x,y\in\Rd$, 
the function
 $\mJ_{\alpha}(t,x,y)$
is the fundamental solution of the equation $\partial_t  =\Delta-2\alpha\cdot \nabla$, i.e.,
$$\mJ_{\alpha}(t,x,y)= g(t,x-2\alpha t,y)\,.$$
We will show that $e_*$ is also comparable with $S$ and $N$. To this end we will use
$$
\dJ_*(V,t)= \sup_{\alpha, x\in\Rd} \int_0^{t}\int_{\Rd}  \mJ_{\alpha}(s,x,z) |V(z)|\,dzds\,.
$$

\begin{lem}\label{lem:N_J}
For all $t>0$ and $V$ we have
\begin{align*}
 \dJ_*(V,t/2) \leq (4\pi)^{-d/2}\|N(V,t)\|_{\infty}
 \leq 2\, \dJ_*(V,t/2)\,.
\end{align*}
\end{lem}
\begin{proof}
Note that
\begin{align*}
\sup_{x,y\in\Rd} \int_0^{t/2}\int_{\Rd}\frac{e^{-|z-y+(\tau/t)(y-x)|^2/(4\tau)}}{\tau^{d/2}}|V(z)|\,dzd\tau
&=\sup_{\alpha, x\in\Rd} \int_0^{t/2}\int_{\Rd}\frac{e^{-|z-x+ 2\alpha \tau|^2/(4\tau)}}{\tau^{d/2}}|V(z)|\,dzd\tau\\
&=(4\pi)^{d/2} \sup_{\alpha, x\in\Rd}
\int_0^{t/2}\int_{\Rd}  \mJ_{\alpha}(\tau,x,z) |V(z)|\,dzd\tau\,.
\end{align*}
The assertion of the lemma follows from \cite[Lemma~3.1]{MR3914946}.
\end{proof}

\begin{lem}\label{lem:R_J}
For all $\lambda>0$, $\alpha\in\Rd$ and $V$ we have
\begin{align*}
(1-e^{-1})\, \|  (\lambda -\Delta +2\alpha\cdot \nabla )^{-1}  |V| \|_{\infty}
&\leq 
\sup_{x\in\Rd} \int_0^{1/\lambda}\int_{\Rd}  \mJ_{\alpha}(s,x,z) |V(z)|\,dzds\,,\\
e\, \|  (\lambda -\Delta +2\alpha \cdot \nabla )^{-1}  |V| \|_{\infty}
&\geq \sup_{x\in\Rd} \int_0^{1/\lambda}\int_{\Rd}  \mJ_{\alpha}(s,x,z) |V(z)|\,dzds\,.
\end{align*}
\end{lem}
\begin{proof}
For $t>0$, $x\in\Rd$ we let 
$P_t f(x)=\int_{\Rd} \mJ_{\alpha}(t,x,z)f(z)\,dz$. Note that
$$
\|  (\lambda -\Delta +2\alpha\cdot \nabla )^{-1}  |V| \|_{\infty}
%=\sup_{x\in\Rd}  \int_0^{\infty} \int_{\Rd} e^{-\lambda t} \mJ_{\alpha}(t,x,z) |V(z)|\,dzdt
 = \sup_{x\in\Rd} \int_0^{\infty} e^{-\lambda t}P_t |V|(x)\,dt\,,
$$
and
$$
\sup_{x\in\Rd}\int_0^{1/\lambda}\int_{\Rd}  \mJ_{\alpha}(s,x,z) |V(z)|\,dzds = \sup_{x\in\Rd} \int_0^{1/\lambda} P_t |V| (x)\,dz\,.
$$
Therefore, the desired inequalities follow from \cite[Lemma~3.3]{MR3713578}.
\end{proof}

Recall from \cite[Corollary~2.3]{MR3914946} that 
for all $T>0$ and $V$ we have
\begin{align}\label{ineq:2T-T}
\|S(V)\|_{2T,\infty}\leq 2 \|S(V)\|_{T,\infty}
\end{align}
Now, 
\eqref{L}, \eqref{U}, \eqref{ineq:2T-T}, Lemma~\ref{lem:N_J} and Lemma~\ref{lem:R_J} 
provide the following comparability.
\begin{prop}\label{prop:comp}
For all $T>0$ and $V$ we have
\begin{align}\label{ineq:S_N}
\frac{m_1}{2} \|N(V)\|_{T,\infty} &\leq \|S(V)\|_{T,\infty}\leq m_2 \|N(V) \|_{T,\infty}\,,
\end{align}
as well as
\begin{align}\label{ineq:N_r}
 \dJ_*(V,T/2)  &\leq   (4\pi)^{-d/2}\,\|N(V)\|_{T,\infty} \leq 2 \, \dJ_*(V,T/2)\,,
\end{align}
and
\begin{align}\label{ineq:resolvent_semigroup}
(1-e^{-1})\, e_*(V,1/T) &\leq 
\dJ_*(V,T)
\leq e \,  e_*(V,1/T)\,.
\end{align}
\end{prop}
Thus, from Proposition~\ref{prop:comp}
and 
Lemma~\ref{lem:comb}
we conclude that the four tests on $V$,  for
the local sharp Gaussian estimates \eqref{est:sharp_time} 
to hold,
based on
$S$, $N$, $\dJ_*$ and $e_*$
are equivalent if $V\leq 0$; 
and comparable if $V\geq 0$
(in that case the exact magnitudes of quantities used in those tests matter, see part 3) of Lemma~\ref{lem:comb}).
In this context we highly
recommend the reader
to get familiar with
\cite[Theorem~1B and Theorem~1C]{MR1994762}), where $e_*$ is brought into play.

We end this subsection by one more observation on $S$ and $N$.
Due to 
Lemma~\ref{lem:N_J},
\eqref{ineq:2T-T},
\eqref{ineq:S_N}
and
\eqref{L}
the supremum over $0<t\leq T$
in $\|S(V)\|_{T,\infty}$ and $\|N(V)\|_{T,\infty}$  is, in a sense, dispensable.
\begin{cor}\label{cor:reduction}
For all $T>0$ and $V$ we have
$$\|N(V,T)\|_{\infty} \leq \|N(V)\|_{T,\infty}\leq 2 \|N(V,T)\|_{\infty}\,,$$
and
$$\|S(V,T)\|_{\infty} \leq \|S(V)\|_{T,\infty}\leq 4(m_2/m_1) \|S(V,T)\|_{\infty}\,.$$
\end{cor}

\subsection{Kato class analogies}

It is well known that $V\in \pK_d$ if and only if
\begin{align*}
\lim_{\lambda \to \infty} \| (\lambda -\Delta)^{-1}|V| \|_{\infty} =0\,.
\end{align*}
Actually, taking $\alpha=0$ in Lemma~\ref{lem:R_J}, for all $\lambda>0$  and $V$ we get
\begin{align*}
(1-e^{-1})\| (\lambda -\Delta)^{-1}|V| \|_{\infty}
\leq 
\sup_{x\in\Rd} \int_0^{1/\lambda}\int_{\Rd}  g(s,x,z) |V(z)|\,dzds
\leq
e
\| (\lambda -\Delta)^{-1}|V| \|_{\infty}\,,
\end{align*}
which is rather a general relation between a semigroup and its resolvent, see \cite[Lemma~3.3]{MR3713578}.
In particular,
$V$ belongs to 
 the enlarged Kato class
if and only if
$\| (\lambda -\Delta)^{-1}|V| \|_{\infty}<\infty$ 
for some (every) $\lambda>0$.
In view of our main discourse 
on sharp Gaussian estimates
a counterpart of those inequalities
is given in
\eqref{ineq:resolvent_semigroup},
also as a consequence of Lemma~\ref{lem:R_J}.

The following result leads to an alternative description of the Kato class
(see \cite[Theorem~1.27]{MR1772266}).
\begin{lem}\label{lem:Katodescr}
There are constants $C_1$ and $C_2$ 
that depend only on dimension $d$ and
such that for all $t>0$ and $V$ we have 
\begin{align}
C_1 A(t) &\le \left[\sup_{x\in\Rd} \int_{|z-x|<\sqrt{4t}} \frac{|V(z)|}{|z-x|^{d-2}}\,dz\right] \le C_2 A(t)\,,&\qquad d\geq 3; \label{est:A3} \\
C_1 A(t) &\le  \left[\sup_{x\in\RR^2} \int_{|z-x|<\sqrt{4t}} |V(z)|\log\frac{4t}{|z-x|^2} \,dz\right] \le C_2 A(t) \,,&\qquad d=2; \label{est:A2}\\
C_1 A(t) &\le 
\left[  \sup_{x\in\RR} \,  \sqrt{t}\!\!\!  \int_{|z-x|<\sqrt{4t}}  |V(z)|\,dz\right] \le C_2 A(t) \,, &\qquad d=1; \label{est:A1}
\end{align}
where
\begin{align*}
A(t) = \sup_{x\in\Rd} \int_0^{t}\int_{\Rd}  g(s,x,z)|V(z)|\,dzds\,.
\end{align*}
\end{lem}

\begin{proof}
First note that the heat kernel $p_{0,d}$ defined in  
\cite[page~47]{MR1772266} has a different time scaling than~$g$, i.e., $g(t,x,y)=p_{0,d}(2t,x,y)$ and $\int_0^t\int_{\Rd} g(s,x,z)|V(z)|dzds=\frac12\int_0^{2t}\int_{\Rd} p_{0,d}(s,x,z)|V(z)|dzds$.
The inequalities \eqref{est:A3}
are now deduced from
\cite[Theorem~1.28(a)]{MR1772266}.
The upper bound in \eqref{est:A2}
follows from the lower bound in
\cite[Theorem~1.28(b)]{MR1772266}.
To prove the lower bound in 
\eqref{est:A2} we note that
\begin{align*}
\int_{|z-x|< \sqrt{4t}} |V(z)|dz
&=\int_{|z-x|< \frac32 \sqrt{t}} |V(z)|dz+
\int_{\frac32 \sqrt{t} \leq |z-x|<2 \sqrt{t}} |V(z)|dz\\
&\leq 5 \sup_{x\in\RR^2} \int_{|z-x|<\frac32 \sqrt{t}} |V(z)|dz
\leq \frac{5}{2\log(4/3)}  \sup_{x\in\RR^2} \int_{|z-x|<\frac32 \sqrt{t}} |V(z)| \log \frac{4t}{|z-x|^2} dz\,,
\end{align*}
and apply the upper bound in
\cite[Theorem~1.28(b)]{MR1772266}.
Finally we look at \eqref{est:A1},
and due to \cite[Theorem~1.28(b)]{MR1772266}
it suffices to show the upper bound in \eqref{est:A1}.
To this end we observe that
\begin{align*}
\int_{|z-x|<\sqrt{4t}} \left(\sqrt{4t}-|z-x|\right)|V(z)|dz
+\int_{|z-x|<\sqrt{4t}}|z-x||V(z)|dz &\\
 \geq
\int_{|z-x|<\sqrt{t}} \sqrt{t}\, |V(z)|dz
+\int_{\sqrt{t} \leq |z-x|<\sqrt{4t}}\sqrt{t}|V(z)|dz 
& \geq
\sqrt{t} \int_{|z-x|<\sqrt{4t}}|V(z)|dz\,,
\end{align*}
and use the lower bound in \cite[Theorem~1.28(c)]{MR1772266}.
\end{proof}

Therefore, $V$ belongs to the Kato class if the expressions in the square brackets 
of Lemma~\ref{lem:Katodescr}
converge to 0, see also
\cite[Theorem~4.5]{MR644024},
\cite[Proposition~A.2.6]{MR670130},
\cite[Theorem~3.6]{MR1329992},
\cite[Proposition~4.3]{MR3914946}.
In Section~\ref{sec:new_test}  we establish a counterpart of Lemma~\ref{lem:Katodescr} describing sharp Gaussian estimates \eqref{est:sharp_time}.

At least in high dimensions the latter description of the Kato class may be viewed through the prism of the following property:
for every $d\geq 3$ there exists a constant $c>0$ that depends only on $d$ and such that for all
$t>0$, 
$x,z\in\Rd$ satisfying
$|z-x|\leq \sqrt{4t}$ we have
\begin{align*}%\label{ineq:t_vs_infty-kato}
c^{-1} \int_0^{\infty} g(s,x,z)\,ds
\leq \int_0^{t} g(s,x,z)\,ds
\leq \int_0^{\infty} g(s,x,z)\,ds\,.
\end{align*}
In the context of sharp Gaussian estimates
an analogue of that observation is proven in
Proposition~\ref{thm:J_K_new}, more precisely in
\eqref{ineq:t_vs_infty}.

\section{A new test for sharp bounds}\label{sec:new_test}

Each of the tests based on $S$, $N$, $\dJ_*$ or $e_*$
may have 
various
advantages and disadvantages when applying to particular functions $V$. 
The utility of 
the condition based on $S$
has already been exposed in
\cite[Section~1.2]{MR3914946}
for functions $V$ that factorize.
We use this paper as an opportunity
to propose another equivalent test based on a  function $K(t,x,y)$, which originates in $\dJ_*(V,T)$.
More precisely, we will estimate $\dJ_*(V,T)$ by investigating the kernel
$\int_0^T \mJ_{\alpha}(s,x,z)ds$ on a certain crucial region.
In what follows the notation is chosen to be consistent with \cite{MR3914946}.
For $t>0$, $x,y\in\Rd$
we let:\\
\begin{align*}
K(t,x,y)&=
e^{-\frac{|x||y|-\left<x,y\right>}{2}} \dfrac{1}{|x|^{d-2}} \left(1+|x||y|\right)^{d/2-3/2}\ind_{|x|\leq t|y|}\,,&
{\rm if}\quad d\geq 3;\\
&&\\
K(t,x,y)&=
e^{-\frac{|x||y|-\left<x,y\right>}{2}} \log\left( 1+\dfrac{1}{\sqrt{|x||y|}}\right)\ind_{|x|\leq t|y|}\,,&
{\rm if}\quad d= 2;\\
&&\\
K(t,x,y)&=
e^{-\frac{|x||y|-\left<x,y\right>}{2}} 
\sqrt{t}\left(1+t|y|^2\right)^{-1/2}
\ind_{|x|\leq t|y|}\,,&
{\rm if}\quad d=1.\\
\end{align*}
We further define 
\begin{align*}
K(V,t,x,y)=\int_{\Rd} K(t,z-x,y) |V(z)|\,dz\,,
\qquad\qquad
\|K(V,t)\|_{\infty}=\sup_{x,y\in\Rd} K(V,t,x,y)\,.
\end{align*}

\begin{thm}\label{thm:new_equiv}
There are constants $0<C_1<C_2<\infty$ that depend only on $d$ and such that
for all $T>0$ and $V$ we have
$$
C_1  \|K(V,T)\|_{\infty} \leq  \|S(V)\|_{T,\infty} \leq  C_2 \|K(V,T)\|_{\infty}\,.
$$
\end{thm}

Before giving the proof of Theorem~\ref{thm:new_equiv}
we provide
consequences, comments and auxiliary results.

\begin{cor}
Let $V\leq 0$.
Then \eqref{est:sharp_time}
holds if and only if $\|K(V,T)\|_{\infty}<\infty$
for some (every) $T>0$.
\end{cor}

\begin{rem}
If $d\geq 3$, using Proposition~\ref{prop:comp}, Theorem~\ref{thm:new_equiv} and letting  $T\to \infty$ we recover the result of \cite[Theorem~1.4]{MR3914946} that concerns global sharp Gaussian estimates \eqref{est:sharp_uni}.
\end{rem}

We note that the kernels of $S$ and $N$ are given explicitly, but 
they are of rather complex structure that involve three parameters 
$0<t\leq T$, $x,y\in\Rd$
that the supremum is taken of.
Corollary~\ref{cor:reduction}  makes it possible to remove one parameter from $S$ and $N$.
Certain reduction is also made in
$e_*$ and $\dJ_*$, where only two parameters $\alpha,x\in\Rd$ appear.
It is also known and results from a  simple substitution
(see \cite[8.432, formula~6.]{MR3307944}) that for $\lambda>0$
and $x,z,\alpha\in\Rd$,
\begin{align}\label{eq:rez-ker}
\int_0^{\infty}e^{-\lambda s} \mJ_{\alpha}(s,x,z)\,ds=
(2\pi)^{-d/2}
 e^{-\left<z-x,  \alpha\right>} 
\left(\frac{\sqrt{\lambda+|\alpha|^2}}{|z-x|}\right)^{d/2-1}
 K_{d/2-1}\left(|z-x|\sqrt{\lambda+|\alpha|^2}\right),
\end{align}
where $K_{\nu}$ is the modified Bessel function of the second kind.
Thus,
\begin{align*}
e_*(V,\lambda)=
(2\pi)^{-d/2}
\sup_{\alpha,x\in\Rd}
\int_{\Rd} e^{-\left<z-x,  \alpha\right>} 
\left(\frac{\sqrt{\lambda+|\alpha|^2}}{|z-x|}\right)^{d/2-1}
 K_{d/2-1}\left(|z-x|\sqrt{\lambda+|\alpha|^2}\right) |V(z)|\,dz\,.
\end{align*}
It is well known that $K_{d/2-1}$ admits the following estimates  
$K_{d/2-1} \approx z^{1-d/2} e^{-z} (1+z)^{d/2-3/2}$, $d\geq 3$,
$K_0 \approx \ln(1+z^{-1/2})e^{-z}$
(see \cite[formulas 9.6.6, 9.6.8, 9.6.9, 9.7.2]{MR0167642}, \cite[page 11]{MR3914946})
and additionally
$K_{-1/2}(z) =  \sqrt{2/\pi} e^{-z} z^{-1/2}$
(see \cite[formula 10.2.16, 10.2.17]{MR0167642}).
Hence,
\begin{align*}
e_*(V,\lambda)
&\approx
\sup_{\alpha,x\in\Rd}
\int_{\Rd}\frac{e^{-\left<z,\alpha\right>-|z|\sqrt{\lambda+|\alpha|^2}}}{|z|^{d-2}}
\left(1+|z|\sqrt{\lambda+|\alpha|^2}\right)^{d/2-3/2}|V(z+x)|\,dz \,,
& {\rm if}\quad  d\geq 3;\\
&&\\
e_*(V,\lambda)&\approx
\sup_{\alpha,x\in\RR^2} \int_{\RR^2} e^{-\left<z,\alpha\right>-|z|\sqrt{\lambda+|\alpha|^2}} \log\left(1+\left({|z|\sqrt{\lambda+|\alpha|^2}}\right)^{-1/2}\right)
|V(z+x)|\,dz\,,
& {\rm if} \quad d=2;\\
&&\\
e_*(V,\lambda)&=
\sup_{\alpha,x\in\RR} 
\frac{1}{2}
\int_{\RR} e^{-\left<z,\alpha\right>-|z|\sqrt{\lambda+|\alpha|^2}} \left( \sqrt{\lambda+|\alpha|^2}\right)^{-1} |V(z+x)|\,dz\,,
& {\rm if} \quad d=1.\\
\end{align*}
Here $\approx$ means that the ratio of both sides is bounded above and below by positive constants independent of $\lambda$ and $V$.
Actually, the comparability constants in the above depend only on $d$.

The relation between the exponents
of the kernel $K(t,x,y)$ and in the explicit estimates of 
$e_*$
becomes more visible 
when putting
$\alpha=-y/2$
and after noticing that
\begin{align}\label{eq:r_s}
\left<z,\alpha\right>+|z|\sqrt{\lambda+|\alpha|^2}
=\left<z,\alpha\right>+|z||\alpha|+
|z|\frac{\lambda}{\sqrt{\lambda+|\alpha|^2}+|\alpha|}\,.
\end{align}
What is more, 
on its support
$K(t,x,y)$
coincides with the above explicit estimates of $e_*$ with $\lambda=0$ if $d\geq 2$,
and a similar comparability holds with $\lambda=1/t$ if $d=1$.
This is not a coincidence and it becomes clear 
by the next proposition,
which plays a key role in the proof of Theorem~\ref{thm:new_equiv}
and which reveals the origin of the function $K(t,x,y)$.

\begin{prop}\label{thm:J_K_new}
For all $t>0$, 
$\alpha,x,z\in\Rd$ satisfying
$|z-x|\leq 2|\alpha| t$ we have
\begin{align}
\frac12 \int_0^{\infty} \mJ_{\alpha}(s,x,z)\,ds
&\leq \int_0^{t} \mJ_{\alpha}(s,x,z)\,ds
\leq \int_0^{\infty} \mJ_{\alpha}(s,x,z)\,ds\,, \qquad d\geq 2 ;
\label{ineq:t_vs_infty} \\
\frac{e}{e+1}\int_0^{\infty} e^{-s/t} \mJ_{\alpha}(s,x,z)\,ds
&\leq \int_0^{t} \mJ_{\alpha}(s,x,z)\,ds
\leq e \int_0^{\infty} e^{-s/t} \mJ_{\alpha}(s,x,z)\,ds\,,\qquad d=1.
\label{ineq:t_vs_e-infty}
\end{align}
There are constants $0<n_1\leq n_2<\infty$ that depend only on $d$ and such that 
for all $t>0$, 
$\alpha,x,z\in\Rd$ satisfying
$|z-x|\leq 2|\alpha| t$ we have
\begin{align}\label{ineq:J_K_fun}
n_1 K(t,z-x,-2\alpha) \leq 
\int_0^{t} \mJ_{\alpha}(s,x,z)\,ds
\leq  n_2 K(t,z-x,-2\alpha)\,.
\end{align}
\end{prop}

\begin{proof}
For simplicity we let $\tilde{x}=z-x$ and $y=-2\alpha$. Then
we have 
\begin{align*}
\int_0^{t} \mJ_{\alpha}(s,x,z)\,ds
= (4\pi t)^{-d/2} \,t \int_0^1 s^{-d/2} e^{-\frac{|\tilde{x}-ts y|^2}{4ts}} \,ds\,.
\end{align*}
Since for $|\tilde{x}|\leq |y| t$ and $s\in (0,1)$, we have
$$\frac{|\tilde{x}|^2}{s} + s |ty|^2\leq \frac{|ty|^2}{s} +s |\tilde{x}|^2\,.$$ 
For $d\geq 2$ we get
\begin{align*}
\int_0^1 s^{-d/2} e^{-\frac{|\tilde{x}-ts y|^2}{4ts}} \,ds
&=e^{\frac{\left<\tilde{x},y\right>}{2}} \int_0^1 s^{-d/2+1} e^{-\left(\frac{|\tilde{x}|^2}{s} +s |ty|^2\right)/(4t)}\, \frac{ds}{s}\\
&\geq e^{\frac{\left<\tilde{x},y\right>}{2}} \int_0^1 s^{d/2-1} e^{-\left(\frac{|ty|^2}{s} +s |\tilde{x}|^2\right)/(4t)} \,\frac{ds}{s}
= \int_1^{\infty} u^{-d/2+1} e^{-\frac{|\tilde{x}- tu y|^2}{4t u}} \,\frac{du}{u}\,.
\end{align*}
Therefore, for $|z-x|\leq 2|\alpha| t$,
\begin{align*}
 \int_0^{\infty} \mJ_{\alpha}(s,x,z)\,ds
\leq 2 \int_0^{t} \mJ_{\alpha}(s,x,z)\,ds\,.
\end{align*}
This proves \eqref{ineq:t_vs_infty}. 
For $d=1$ we have
\begin{align*}
\int_0^1 s^{-1/2} e^{-\frac{|\tilde{x}-ts y|^2}{4ts}} \,ds
&=e^{\frac{\left<\tilde{x},y\right>}{2}} \int_0^1 s^{1/2} e^{-\left(\frac{|\tilde{x}|^2}{s} +s |ty|^2\right)/(4t)}\, \frac{ds}{s}
\geq e^{\frac{\left<\tilde{x},y\right>}{2}} \int_0^1 s^{1/2} e^{-\left(\frac{|ty|^2}{s} +s |\tilde{x}|^2\right)/(4t)} \,\frac{ds}{s}\\
&= \int_1^{\infty} u^{-1/2} e^{-\frac{|\tilde{x}- tu y|^2}{4t u}} \,\frac{du}{u}
\geq e \int_1^{\infty} e^{-u} u^{1/2} e^{-\frac{|\tilde{x}- tu y|^2}{4t u}} \,\frac{du}{u} \,.
\end{align*}
Therefore, for $|z-x|\leq 2|\alpha| t$,
\begin{align*}
 \int_0^{\infty} e^{-s/t} \mJ_{\alpha}(s,x,z)\,ds
\leq (1+1/e) \int_0^{t} \mJ_{\alpha}(s,x,z)\,ds\,.
\end{align*}
This ends the proof of \eqref{ineq:t_vs_e-infty}.
Now, note that we can take $\lambda=0$ in \eqref{eq:rez-ker} by passing with $\lambda>0$ to zero.
Then \eqref{ineq:J_K_fun}
follows from \eqref{eq:rez-ker} and the 
estimates of $K_{\nu}$ mentioned above;
and from \eqref{eq:r_s} for $d=1$.
\end{proof}

\begin{lem}
For all $T>0$ and $V$ we have
\begin{align}
\dJ_*(V,T) &\geq 
\frac{n_1}{2} \|K(V,T)\|_{\infty}+ \frac{1}{2}\sup_{x\in\Rd} \int_0^{T}\int_{\Rd}  g(s,x,z) |V(z)|\,dzds \,, \label{ineq:r_lower} \\
\dJ_*(V,T) & \leq n_2 \|K(V,T)\|_{\infty}+2^{d-2} \sup_{x\in\Rd} \int_0^{4T}\int_{\Rd}  g(s,x,z) |V(z)|\,dzds\,.
\label{ineq:r_upper}
\end{align}
The constants $0<n_1\leq n_2<\infty$ are taken from \eqref{ineq:J_K_fun}.
\end{lem}
\begin{proof}
Recall that
$\mJ_{\alpha}(t,x,z)= g(t,x-2\alpha t,z)$.
If we put $\alpha=0$, we get
that $\dJ_*(V,T)$ is bounded below by  $\sup_{x\in\Rd} \int_0^{T}\int_{\Rd}  g(s,x,z) |V(z)|\,dzds $,
while by
reducing the domain of  integration in space variable  $z$ to $|z-x|\leq 2|\alpha|t$ and by \eqref{ineq:J_K_fun}
we have
$\dJ_*(V,T) \geq n_1 \|K(V,T)\|_{\infty}$.
That proves the lower bound
\eqref{ineq:r_lower}.
Now, let $y=-2\alpha$.
For the upper bound we consider two regions of integration,
\begin{align*}
A_1&=\{z\in\Rd \colon |z-x|> t|y|\}\,,\\
A_2&=\{z\in\Rd \colon |z-x|\leq t|y|\} \,.
\end{align*}
Note that if $z\in A_1$ and $s\in (0,t)$, then
\begin{align*}
|z-x-ty|&\leq |z-x-s y|+(t-s)|y| \\
&< |z-x-s y|+|z-x|-|s y|\leq 2|z-x-s y|\,.
\end{align*}
By the monotonicity of the exponential function we get
\begin{align*}
\int_0^t\int_{A_1}  g(s,x+sy,z)|V(z)|\,dz ds
&\leq 
2^d\int_0^t\int_{\Rd}  g(4s,x+ty,z)|V(z)|\,dz ds\\
&=2^{d-2}\int_0^{4t}\int_{\Rd}  g(u,x+ty,z)|V(z)|\,dz du\,.
\end{align*}
On the set $A_2$ we apply \eqref{ineq:J_K_fun}. This ends the proof of \eqref{ineq:r_upper}.
\end{proof}

We are now ready to justify Theorem~\ref{thm:new_equiv}.

\begin{proof}[Proof of Theorem~\ref{thm:new_equiv}]
We will actually prove that
$$
c_1  \|K(V,T)\|_{\infty} \leq  \dJ_*(V,T) \leq  c_2 \|K(V,T)\|_{\infty}\,,
$$
for all $T>0$ with constants $0<c_1< c_2<\infty$
that depend only on $d$. The result will then follow from Proposition~\ref{prop:comp} and \eqref{ineq:2T-T}.
The lower bound holds by \eqref{ineq:r_lower}.
We focus on the upper bound
and due to  \eqref{ineq:r_upper}
it suffices to show that
$$
\sup_{x\in\Rd} \int_0^{4T}\int_{\Rd}  g(s,x,z) |V(z)|\,dzds
\leq c\, \|K(V,T)\|_{\infty}\,.
$$
For the whole proof we
let $y=(4 t^{-1/2},0,\ldots,0)\in\Rd$. Then
for $d\geq 3$,
since $-\left<x,y\right>\leq |x||y|$, we have
\begin{align*}
K(t,x,y) \geq e^{-4|x|  t^{-1/2}} \frac{1}{|x|^{d-2}}
\ind_{|x|\leq \sqrt{16t}}
\geq e^{-16} \frac{1}{|x|^{d-2}}\ind_{|x|\leq \sqrt{16t}}\,.
\end{align*}
Therefore, by 
\eqref{est:A3}
(cf. \cite[(4.3)]{MR3200161})
there is a constant $c>0$ that depends only on $d$ such that
\begin{align*}
\|K(V,T)\|_{\infty}
\geq e^{-16} \sup_{x\in\Rd} \int_{|z-x|\leq \sqrt{16T}}
 \frac{|V(z)|}{|z-x|^{d-2}}\,dz
\geq 
c \sup_{x\in\Rd} \int_0^{4T}\int_{\Rd}  g(s,x,z) |V(z)|\,dzds\,.
\end{align*}
For $d=2$ we
first note that 
$\log(1+r/2)\geq (1/3)\log (r)$
if $r\geq 1$.
Therefore, 
\begin{align*}
K(t,x,y)&\geq
e^{-16}
\log\left( 1+\frac1{2}\left(\frac{16 t}{|x|^2}\right)^{1/4}\right)\ind_{|x|\leq \sqrt{16t}}
\geq (e^{-16}/3) \log\left(\frac{16 t}{|x|^2} \right)\ind_{|x|\leq \sqrt{16t}}\,.
\end{align*}
Finally, by 
\eqref{est:A2}
there is an absolute constant $c>0$ such that
\begin{align*}
\|K(V,T)\|_{\infty}
&\geq (e^{-16}/3)
\sup_{x\in\RR^2}  \int_{|z-x|\leq \sqrt{16T}}
\log\left(\frac{16T}{|z-x|^2}\right)|V(z)|\,dz
\geq c \sup_{x\in\RR^2} \int_0^{4T}\int_{\RR^2}  g(s,x,z) |V(z)|\,dzds\,.
\end{align*}
For $d=1$ we have
\begin{align*}
K(t,x,y)
\geq \frac{e^{-16}}{\sqrt{17}} \sqrt{t} \,\ind_{|x|\leq \sqrt{16t}}\,,
\end{align*}
and by \eqref{est:A1}
there is an absolute constant $c>0$ such that
\begin{align*}
\|K(V,T)\|_{\infty}
&\geq \frac{e^{-16}}{\sqrt{17}} \sup_{x\in\RR} 
\,  \sqrt{T}\!\!\!  \int_{|z-x|<\sqrt{16T}}  |V(z)|\,dz 
\geq c \sup_{x\in\RR} \int_0^{4T}\int_{\RR}  g(s,x,z) |V(z)|\,dzds\,.
\end{align*}
\end{proof}

\section{Proofs of Theorems \ref{thm:dgeq4} -- \ref{thm:d2d1}}\label{sec:proofs}

\subsection{Proof of Theorem~\ref{thm:dgeq4}}

In the proof we construct a function $V$ with the desired properties. The construction is based on another function defined in \cite[Proposition~1.6]{MR3914946}, and uses truncations and dilatations.

\begin{proof}
For $s>0$ we let $\tau_s f(x)=sf(\sqrt{s}x)$. Note that such dilatation does not change the norm
$$
\| \Delta^{-1}(\tau_s f)  \|_{\infty}=\| \Delta^{-1} f \|_{\infty}\,.
$$
Moreover, ${\rm supp} (\tau_s f) \subseteq B(0,r/\sqrt{s})$ if ${\rm supp} (f)\subseteq B(0,r)$, $r>0$, and for $t>0$,
$$
\|S(\tau_s f,t)\|_{\infty}=\|S(f,st)\|_{\infty}\,.
$$
Now, let $U\colon \Rd\to \RR$ be non-positive and such  that
$$
\|U\|_\infty\leq 1\,,\quad \|\Delta^{-1}U\|_{\infty}=C<\infty\,,\quad \sup_{t>0}\|S(U,t)\|_{\infty}=\infty\,.
$$
Such $U$ exists by \cite[Proposition~1.6 and Theorem~1.4]{MR3914946}. 
By the definition of the supremum norm and the monotone convergence theorem, 
for $n\in\NN$ there are $t_n, r_n>0$ such that
$
\|S(U\ind_{B_{r_n}},t_n)\|_{\infty}> (4m_2/m_1)\, 4^n 
$.
For simplicity we define $U_n=U\ind_{B_{r_n}}$, so we have
$$
\|S(U_n,t_n)\|_{\infty}> (4m_2/m_1)\, 4^n \,.
$$
Let $s_n=\max\{r_n^2, n\,t_n\}$
and define 
$$
V_n=\tau_{s_n} (U_n)\,.
$$
Then ${\rm supp} (V_n)\subseteq B(0,1)$, $V_n\in L^{\infty}(\Rd)$,  $\|\Delta^{-1} V_n\|_{\infty}\leq C$ and
by Corollary~\ref{cor:reduction}
\begin{align*}
\|S(V_n,1/n)\|_{\infty}
&=\| S(U_n,s_n/n)\|_{\infty}\\
&\geq \left(\frac{m_1}{4m_2}\right)  \|S(U_n)\|_{s_n/n,\infty}\\
&\geq \left(\frac{m_1}{4m_2}\right)  \|S(U_n)\|_{t_n,\infty}
\geq \left(\frac{m_1}{4m_2}\right) \|S(U_n,t_n)\|_{\infty}> 4^n\,.
\end{align*}
Finally, let 
$$V:=\sum_{n=1}^{\infty} V_n/2^n\,.$$ Obviously, part (a) holds. Further,
again by Corollary~\ref{cor:reduction}, for $t>0$ we get
\begin{align*}
\|S(V,t)\|_{\infty}
&\geq \left(\frac{m_1}{4m_2}\right)
\lim_{n\to\infty}\|S(V,1/n)\|_{\infty}
\geq \left(\frac{m_1}{4m_2}\right)
\lim_{n\to\infty}
2^{-n} \|S(V_n,1/n)\|_{\infty}
=\infty\,.
\end{align*}
This proves part (d). 
The statement (c) holds by
$$
\|\Delta^{-1} V\|_{\infty} \leq \sum_{n=1}^{\infty} \|\Delta^{-1}V_n\|_{\infty}/2^n\leq C<\infty\,.
$$
Next,
\begin{align*}
\sup_{x\in\Rd}&\int_0^t \int_{\Rd} 
 g(s,x,z)|V(z)|\,dzds\\
&\leq 
\sum_{n=1}^N  
\sup_{x\in\Rd}\int_0^t \int_{\Rd} 
 g(s,x,z)\frac{|V_n(z)|}{2^n}\,dzds
+\sum_{n=N+1}^{\infty} \sup_{x\in\Rd}\int_0^{\infty} \int_{\Rd} g(s,x,z)\frac{|V_n(z)|}{2^n}\,dzds\\
&\leq \sum_{n=1}^N  t \|V_n\|_{\infty}
+\sum_{n=N+1}^{\infty} \|\Delta^{-1}V_n\|_{\infty}/2^n\\
&\leq
 t \sum_{n=1}^N  \|V_n\|_{\infty}
+\frac{C}{2^N} \,,
\end{align*}
which can be made arbitrary small by the choice of $N$ and $t$,
and proves part (b).

\end{proof}

\subsection{Proof of Theorem~\ref{thm:d3}}

Similarly to the proof of Theorem~\ref{thm:dgeq4} we construct a function $V$ with the desired features. 
We will choose a decreasing function $f\geq 0$ satisfying $\int_0^{1/25}f(r)dr=\infty$.
The function $V$ will be given by a series based on certain functions $V_n$. Each $V_n$ will be supported on a union of properly chosen cylinders $C_{k,r}$ and will have values according to the function $f$.
In particular, the choice will be such that on the support of $V_n$, the function
$K(t,x,y)$ with $|y|=25n$
will be comparable to
$
1/|x|
$ 
and 
such that for a sequence $n_i\in\NN$
diverging to infinity
 we will have
$$
\|K(V_{n_i},1)\|_{\infty} \geq c \int_{1/(25n_i)}^{1/25}f(r)dr\geq  4^i \,.
$$

In the first lemma we investigate a function $U_r$ that is supported on a cylinder $C_r\subset \RR^3$ and takes values related to the size of the cylinder.
To simplify the notation, for
$z=(z_1, z_2,z_3)\in\RR^3$ we write $z=(z_1,\mathbf z_2)$, 
where $\mathbf z_2=(z_2,z_3)\in \RR^2$.

\begin{lem}\label{lem:kato}
For $r>0$ we define
\begin{align*}
C_r = \left[0,\frac14 \right] \times D_r,
\end{align*}
where $D_r$ is a 2-dimensional ball of radius $r$ centred at $0$.
For $r\in(0,e^{-1})$, $z\in\RR^3$ put
$$
\varrho(r)=\frac{1}{r^2 \,|\ln r| \, \ln | \ln r|} \qquad {and}\qquad U_r(z) = \varrho(r) \ind_{C_r}(z)\,. 
$$
Then
\begin{align*}%\label{kato}
\lim_{\varepsilon \to 0^+}  \sup_{
\substack{x\in\RR^3\\ r\in (0,1/5)}} \quad\int_{|z-x| < \varepsilon} \frac{1}{|z-x|} | U_r(z)| dz
=0\,.
\end{align*}
\end{lem}
\begin{proof}
Note that $\varrho(r)$ is decreasing on $(0,1/5)$.
On the other hand, $r^2 \varrho(r)$ and $r^2 |\ln r| \varrho(r)$ are increasing on $(0,1/5)$.
Let $0<\varepsilon<1/5$ and
\begin{align*}
{\rm I}_r(\varepsilon):=\sup_{x\in\RR^3}\int_{|z-x| < \varepsilon} \frac{1}{|z-x|} | U_r(z)| dz
=\varrho(r)  \sup_{x\in\RR^3} \int_{|z|< \varepsilon} \frac{1}{|z|}\ind_{C_r}(z+x) dz\,.
\end{align*}
If $\varepsilon \le r$, then
$$
{\rm I}_r(\varepsilon)
\le \varrho(r)
\int_{|z|< \varepsilon} \frac{1}{|z|}\,dz
\le 2\pi  \varepsilon^2 \varrho(\varepsilon)\,.
$$
If $r \le \varepsilon$, 
we use the symmetric rearrangement inequality \cite[Chapter~3]{MR1817225} and that $\varepsilon<1/5$ to get
\begin{align*}
\int_{|z|< \varepsilon} \frac{1}{|z|}\ind_{C_r} (z+x) dz
&=\int_{-x_1}^{1/4-x_1} dz_1
\int_{\RR^2} \frac{\ind_{|z|<\varepsilon}}{|z|} \, \ind_{D_r}(\mathbf z_2 + \mathbf x_2) d \mathbf z_2 
\leq \int_{-x_1}^{1/4-x_1} dz_1 
\int_{\RR^2} \frac{\ind_{|z|<\varepsilon}}{|z|}\, \ind_{D_r}(\mathbf z_2) d \mathbf z_2\\
&\leq \int_{-1/4}^{1/4} dz_1\int_{\RR^2} \frac{\ind_{|z|<\varepsilon}}{|z|}\, \ind_{D_r}(\mathbf z_2) d \mathbf z_2
=\int_{|z|<\varepsilon} \frac{1}{|z|}\, \ind_{C_r\cup (-C_r)}(z)dz\,.
\end{align*}
Now note that
$$B(0,\varepsilon)\cap (C_r\cup (-C_r))\subseteq B(0,\sqrt{2}r)\cup ( [r,\varepsilon]\times D_r)
\cup ( [-\varepsilon,-r]\times D_r).$$
Then
\begin{align*}
{\rm I}_r(\varepsilon)
\leq 
\varrho(r) \left(\int_{|z|\leq \sqrt{2}r} \frac1{|z|}\,dz + 2 \int_r^\varepsilon \frac{|D_r|}{z_1}\, dz_1  \right)
&\le \varrho(r) \Big( 4\pi r^2 +  2\pi r^2 |\ln r|\Big)\\
&\le \varrho(\varepsilon) \Big( 4\pi \varepsilon^2 +  2\pi \varepsilon^2 |\ln \varepsilon|\Big).
\end{align*}
Thus
\begin{align*}
\lim_{\varepsilon \to 0^+}  \sup_{
\substack{x\in\RR^3\\ r\in (0,1/5)}} \quad\int_{|z-x| < \varepsilon} \frac{1}{|z-x|} | U_r(z)| dz=
\lim_{\varepsilon \to 0^+}  \sup_{r\in (0,1/5)} {\rm I}_r(\varepsilon)
=0\,.
\end{align*}
\end{proof}

\begin{cor}\label{cor:kato}
For $k\in\NN$ and $r>0$ we
define
\begin{align*}
C_{k,r} = \left[k,k+\frac14 \right] \times D_r,
\end{align*}
where $D_r$ is a 2-dimensional ball of radius $r$ centred at $0$.
For $r\in(0,e^{-1})$, $n\in\NN$ and $z\in\RR^3$ put
$$
f(r)=\frac{1}{r \,|\ln r| \, \ln | \ln r|}
\qquad {and}\qquad
V_{n}(z)= f\left(\frac{z_1}{25n}\right) \sum_{k=1}^{n} \ind_{C_{k,\sqrt{k/(25n)}}} (z)\,.
$$
Then
\begin{align*}%\label{kato2}
\lim_{\varepsilon \to 0^+}  
\sup_{x\in\RR^3,\, n\in \NN}
\quad 
\int_{|z-x| < \varepsilon} \frac{1}{|z-x|} | V_{n}(z)| dz=0\,.
\end{align*}
\end{cor}
\begin{proof}
We use
$\varrho(r)$ and $U_r$
as defined in
Lemma~\ref{lem:kato}.
Note that $f(r)$ is decreasing on 
$(0,e^{-3})$ and $f(r^2)\leq \varrho(r)$ on $(0,e^{-1})$.
We record that  every two cylinders $C_{k,\sqrt{k/(25n)}}$
that correspond to different values of $k\in\NN$
are disjoint. Therefore
if $z\in C_{k,\sqrt{k/(25n)}}$
we have
$$
V_{n}(z)
=f\left(\frac{z_1}{25n}\right)
\leq
f\left(\frac{k}{25n}\right)
\leq \varrho
\left(\sqrt{\frac{k}{25n}}\right)
= U_{\sqrt{k/(25n)}} (z-(k,\mathbf 0)).
$$
What is more, the distance between
every two cylinders $C_{k,\sqrt{k/(25n)}}$
that correspond to different $k$
is at least $3/4$.
Thus, 
for any $x\in\Rd$ and $0<\varepsilon<3/8$,
the intersection of $B(x,\varepsilon)$
and ${\rm supp}(V_n)$ is a subset of at most one
cylinder $C_{k,\sqrt{k/(25n)}}$,
and so by Lemma~\ref{lem:kato},
\begin{align*}
\lim_{\varepsilon \to 0^+}  
\sup_{x\in\RR^3,\, n\in \NN}
\,
\int_{|z-x| < \varepsilon} \frac{1}{|z-x|} | V_{n}(z)| dz
&\leq \lim_{\varepsilon \to 0^+} 
\sup_{\substack{x\in\RR^3,\,n\in \NN \\k =1,\ldots,n}}
\,
\int_{|z-x| < \varepsilon} \frac{1}{|z-x|}
 U_{\sqrt{k/(25n)}} (z-(k,\mathbf 0))dz\\
& \leq \lim_{\varepsilon \to 0^+} 
\sup_{\substack{x\in\RR^3\\ r\in (0,1/5)}} \quad\int_{|z-x| < \varepsilon} \frac{1}{|z-x|} | U_r(z)| dz
=0\,.
\end{align*}
\end{proof}

\begin{lem}\label{lem:nosharp}
Let $V_n$ be defined as in Corollary~\ref{cor:kato}.
There are $n_i\in\NN$, $i\in\NN$, such that
for every $i\in\NN$,
$$
\|K(V_{n_i},1)\|_{\infty}\geq  4^i \,.
$$
\end{lem}
\begin{proof}
Let $\theta>0$. Then
\begin{align*}
\theta(|z| - z_1) <1 \qquad \iff  \qquad z_1 > \frac{\theta}{2} |\mathbf z_2|^2 - \frac{1}{2\theta}.
\end{align*}
For $n\in\NN$ we put
\begin{align*}
E_n: = \left\{z \in \RR^3 \colon z_1 > \frac{25 n}{2} |\mathbf z_2|^2 - \frac{1}{50n}\right\}.
\end{align*}
Thus,
for
$
z\in E_n
$
we have
$
25n(|z| - z_1)<1
$.
 Then, by taking $x=0$ and $y = (25n,\mathbf 0)$ in the first inequality below,
and using  ${\rm supp} (V_n)\subset E_{n} \cap B(0,25n)$ in the second one,
\begin{align*}
\|K(V_n,1)\|_{\infty}
&=\sup_{x,y \in \RR^3} \int_{|z-x| \le |y|} 
\frac{e^{-\frac{|z-x||y|-\left<z-x,y\right>}{2}}}{|z-x|} |  V_n(z)| dz\\
&\geq \int_{|z|\le 25n} \frac{e^{-\frac12 \cdot 25n(|z| - z_1)}}{|z|} |V_n(z)| dz 
\geq  e^{-\frac12} \int_{\RR^3}\frac{1}{|z|}|V_n(z)|dz\,.
\end{align*}
Further, by the definition of $V_n$
and $C_{k,\sqrt{k/(25n)}}$,
\begin{align*}
\|K(V_n,1)\|_{\infty}
&\geq e^{-1/2} \sum_{k=1}^{n}\, \int_{\RR^3} \frac1{|z|} f\left(\frac{z_1}{25 n}\right) \ind_{C_{k,\sqrt{k/(25n)}}}(z)dz\\
&\geq e^{-1/2} \sum_{k=1}^{n} \int_k^{k+1/4} \frac1{k+1} f\left(\frac{z_1}{25n}\right)|D_{\sqrt{k/(25n)}}|  dz_1\\
&\geq \frac{\pi e^{-1/2}}{2} \sum_{k=1}^{n} \int_k^{k+1/4} f\left(\frac{z_1}{25n}\right)  
\frac{dz_1}{25n}\\
&\geq \frac{\pi e^{-1/2}}{8}\int_1^{n}
f\left(\frac{z_1}{25n}\right)  
\frac{dz_1}{25n}= \frac{\pi e^{-1/2}}{8} \int_{1/(25n)}^{1/25} f(r)dr.
\end{align*}
This ends the proof since $\int_0^{1/25}f(r)dr=\infty$.
\end{proof}

\begin{proof}[Proof of Theorem~\ref{thm:d3}]
For $n\in\NN$ let $V_n$ be as in Corollary~\ref{cor:kato}
and $(n_i)_{i\in\NN}$ be a sequence of natural numbers taken from Lemma~\ref{lem:nosharp}.
We take
$$
V:=-\sum_{i=1}^{\infty} V_{n_i}/2^i\,.
$$
By Lemma~\ref{lem:nosharp} we have
$$
\|K(V,1)\|_{\infty}\geq \sup_{i\in\NN} 2^{-i}\,\|K(V_{n_i},1)\|_{\infty}=\infty\,.
$$
Therefore, by Theorem~\ref{thm:new_equiv},
\eqref{ineq:2T-T}
and Corollary~\ref{cor:reduction}
part b) follows.
Next, we have
\begin{align*}
\sup_{x\in\RR^3}
\quad 
\int_{|z-x| < \varepsilon} \frac{1}{|z-x|} | V(z)| dz
&\leq 
\sum_{i=1}^{\infty} 2^{-i}
\sup_{x\in\RR^3} 
\quad 
\int_{|z-x| < \varepsilon} \frac{1}{|z-x|} | V_{n_i}(z)| dz\\
&\leq \sup_{x\in\RR^3,\,i\in\NN}
\quad 
\int_{|z-x| < \varepsilon} \frac{1}{|z-x|} | V_{n_i}(z)| dz\,,
\end{align*}
which
can be made arbitrary small by the choice of $\varepsilon$ due to
Corollary~\ref{cor:kato}.
This proves part a).
\end{proof}

\subsection{Proof of Theorem~\ref{thm:d2d1}}

Before we pass to the proof of Theorem~\ref{thm:d2d1} we show the following auxiliary result in $d=2$. For
$z\in\RR^2$ we write 
as usual
$z=(z_1,z_2)$, 
where $z_1,z_2\in \RR$.

\begin{lem}\label{lem:D}
Let $d=2$. For $r \geq 2$ we let 
$D_r=\{z\in\RR^2\colon z_1 \geq 0 \mbox{ and }  2 \leq |z|\leq r\}$.
There exists a constant $c>0$  such that for all Borel measurable $U\colon \RR^2\to [0,\infty]$ and $r\geq 2$,
$$
\int_{D_r} K(1,z,(r,0))\, U(z) dz
\leq c \sup_{w\in\RR^2} \int_{|z|\leq 2} U(z+w)dz\,.
$$
\end{lem}
\begin{proof}
Note that for $r>0$ and $n\in\NN \cup \{0\}$,
$$
r(|z|-z_1)\leq n \quad \iff \quad |z_2|\leq \frac{\sqrt{2 n z_1 r+n^2}}{r}\,.
$$
In the rest of proof we consider $r\geq 2$ and $0\leq z_1\leq r$.
For $n\in\NN \cup \{0\}$ we let
$$
f_n(z_1):=\frac{\sqrt{2nz_1r+n^2}}{r}
\qquad
\mbox{and}
\qquad
F_n:=\{ z\in\RR^2\colon f_n(z_1)\leq |z_2|\leq f_{n+1}(z_1),\,0\leq z_1\leq r\}\,.
$$
Obviously, $f_n$ and $F_n$ depend on $r$, which we do not indicate explicitly to lighten the notation.
In particular, 
$
n \leq r(|z|-z_1) \leq n+1  \iff  z \in F_n
$.
A direct analysis of the derivative shows that
for each $a\geq 0$ and $b>0$ 
a function
$$
h(t)=\sqrt{2(a+b)(t+1)+(t+1)^2}-\sqrt{2at+t^2}\,,\qquad t\geq 0\,,
$$
is decreasing on $[0,a/b]$ and increasing on $[a/b,\infty)$. This guarantees
for each $\delta\in (0,1)$ that
\begin{align*}
f_{n+1}(z_1+\delta)-f_n(z_1) 
&\leq
\max \left\{ f_1(z_1+\delta), \lim_{n\to\infty} (f_{n+1}(z_1+\delta)-f_n(z_1))\right\}\\
&= \max\left\{\frac{\sqrt{2z_1 r+1}}{r},\delta+\frac1r\right\}
\leq \max\left\{\sqrt{2+\frac1{r^2}},\delta+\frac1r\right\}\leq \frac32
\end{align*}
We fix $\delta\in (0,1)$ (any $\delta \leq \sqrt{7}/2$ has that property) so that
for all $n\in\NN \cup \{0\}$,
\begin{align}\label{eq:diag}
\sqrt{(f_{n+1}(z_1+\delta)-f_n(z_1))^2+\delta^2} \leq 2\,.
\end{align}
For $n,k\in\NN \cup \{0\}$ we define rectangles
$$
P_{n,k}:=\Big[ k\delta,\,(k+1)\delta\Big]\times \Big[f_n(k\delta),\,f_{n+1}((k+1)\delta)\Big]\subset \RR^2\,.
$$
\begin{figure}[h]
%\begin{center}
\includegraphics{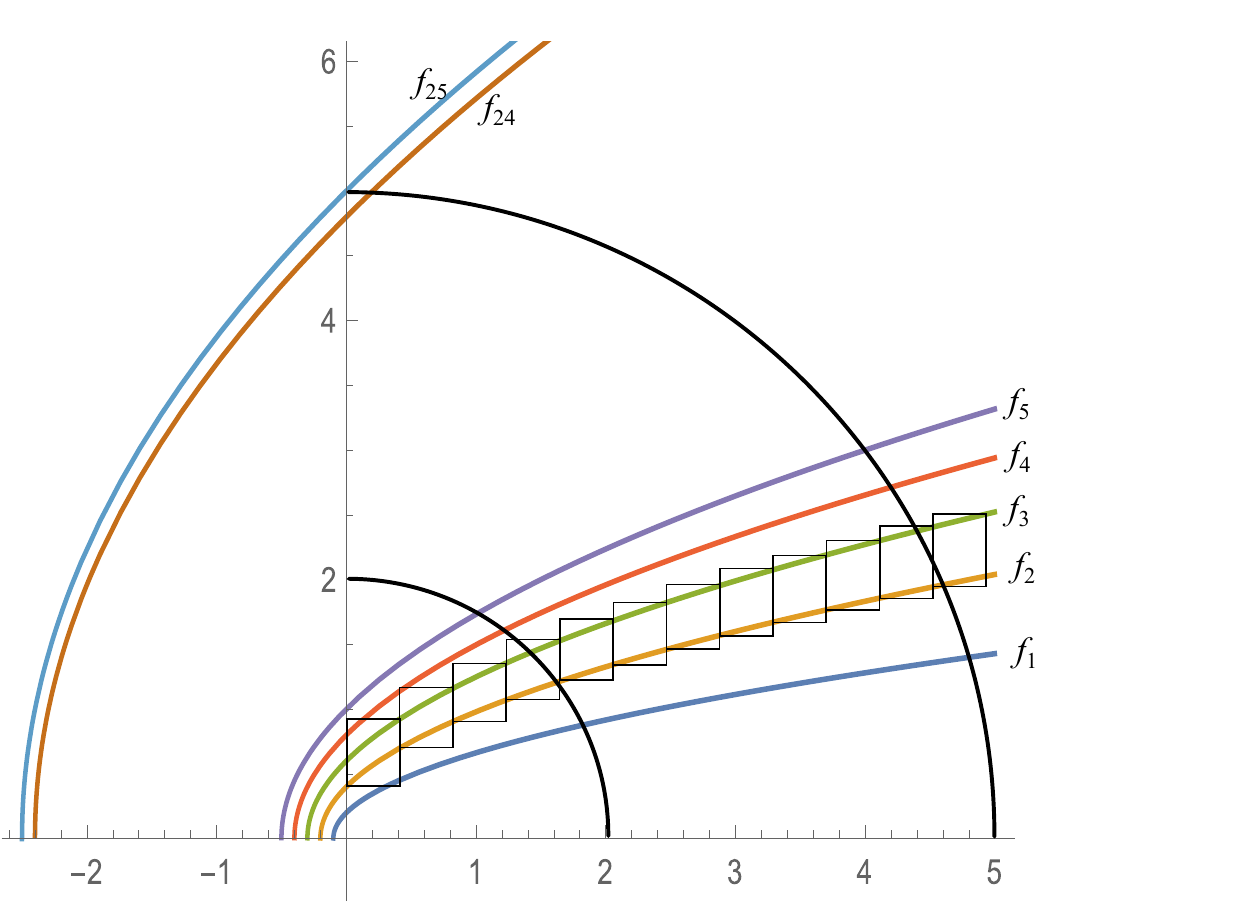}
\centering
%\end{center}
\caption{Graphs of functions $f_n$ and rectangles $P_{i,k}$ for $r=5$, $i=2$ and $\delta =2/5$. Here $F_2 \subset \bigcup\limits_{k=0}^{11} P_{2,k}$.}
\end{figure}

\noindent
The bottom left vertex of $P_{n,k}$ equals $a_{n,k}=(k\delta, f_n(k\delta))$
and satisfies $|a_{n,k}|=k\delta+\frac{n}{r}$. Furthermore, 
if~$k\leq \lfloor r/\delta \rfloor$, then $k\delta\leq r$ and by \eqref{eq:diag} the diagonal of $P_{n,k}$ does not exceed 2. Hence
$P_{n,k}\subseteq B(a_{n,k},2)$, where the latter is 
a 2-dimensional ball of radius $2$ centred at $a_{n,k}$.
Next, observe that
$$
D_r\subseteq   \bigcup_{n=0}^{\lfloor r^2 \rfloor} F_n %\cap \{z\in\RR^2\colon |z|\geq 2 \} 
\qquad
\mbox{and}
\qquad
F_n\subseteq \bigcup_{k=0}^{\lfloor \frac{r}{\delta} \rfloor} P_{n,k}\,.
$$
Finally, on $D_r \cap F_n \cap P_{n,k}$ we have
\begin{align*}
K(1,z,(r,0)) &= e^{-\frac{1}{2} r(|z|-z_1)}\log\left(1+\frac{1}{\sqrt{r|z|}}\right)\ind_{|z|\leq r}\\
&\leq
\frac{e^{-\frac{1}{2} r(|z|-z_1)}}{\sqrt{r|z|}}
\leq 
\frac{e^{-n/2}}{\sqrt{r\max\{|a_{n,k}|,2\}}}
\leq 
\frac{e^{-n/2}}{\sqrt{r(k\delta/2+1)}} \ind_{B(a_{n,k},2)}(z)\,.
\end{align*}
This implies
\begin{align*}
\int_{D_r \cap F_n \cap P_{n,k}} K(1,z,(r,0))\, U(z) dz
&\leq 
\frac{e^{-n/2}}{\sqrt{r(k\delta/2+1)}}
 \int_{|z|\leq 2} U(z+a_{n,k})dz \\
&\leq 
\frac{e^{-n/2}}{\sqrt{r(k\delta/2+1)}} \sup_{w \in \RR^2}
 \int_{|z|\leq 2} U(z+w)dz\,.
\end{align*}
It remains to notice that
$\frac{1}{\sqrt{r}}\sum\limits_{k=0}^{\lfloor \frac{r}{\delta} \rfloor}(k\delta/2+1)^{-1/2}\leq 1+ 4/\delta$
and $\sum\limits_{n=0}^{\infty}e^{-n/2}<\infty$.
\end{proof}

\begin{proof}[Proof of Theorem~\ref{thm:d2d1}]

The lower bound in \eqref{ineq:d2d1}
follows immediately from 
\eqref{ineq:2T-T},
\eqref{ineq:S_N} and
\eqref{ineq:N_r}.
We focus on the upper bound.
Due to Theorem~\ref{thm:new_equiv} it suffices to estimate $\|K(V,t)\|_{\infty}$, $t>0$.
First we consider $d=2$.
For $|y|\leq 2 t^{-1/2}$ we have
\begin{align}\label{ineq:D2}
K(t,x,y)\leq \log\left(1+\frac{
\sqrt{t}}{|x|}\right)\ind_{|x|\leq \sqrt{4t}}
\leq \left( 1+ \log \frac{4t}{|x|^2}\right)\ind_{|x|\leq \sqrt{4t}}\,.
\end{align}
Therefore,
by 
\eqref{est:A2}
there is an absolute constant $c>0$ such that
$$
\sup_{|y|\leq 2 t^{-1/2}} \sup_{x\in\RR^2}K(V,t,x,y)\leq c \sup_{x\in\RR^2} \int_0^t\int_{\RR^2} g(s,x,z) |V(z)|dzds\,.
$$
We focus on $|y|\geq 2 t^{-1/2}$.
Let
\begin{align*}
A_1&=\{z\in\RR^2 \colon \left<z-x,y\right> \leq 0\}\,,\\
A_2&=\{z\in\RR^2 \colon \left<z-x,y\right> \geq 0 \mbox{ and } |z-x|\leq \sqrt{4 t}\}\,,\\
A_3&=\{z\in\RR^2 \colon \left<z-x,y\right> \geq 0 \mbox{ and } \sqrt{4t} \leq |z-x|\leq t|y|\} \,.
\end{align*}
On~the set $A_1$ we have $|z-x-sy| \ge |z-x|$, hence
by \eqref{ineq:J_K_fun} we get
$$
n_1 K(t,z-x,y)
\leq \int_0^t p_{(-y/2)}(s,x,z)ds
= \int_0^t g(s,x+sy,z)ds
\leq \int_0^t g(s,x,z)ds\,.
$$
Thus
$$
\sup_{|y|\geq 2 t^{-1/2}} \sup_{x\in\RR^2}\int_{A_1} K(t,z-x,y)|V(z)|dz\leq (1/n_1)\sup_{x\in\RR^2}
\int_0^t\int_{\RR^2} g(s,x,z) |V(z)|dzds\,.
$$
On the set $A_2$ we argue like in \eqref{ineq:D2}, therefore
$$
\sup_{|y|\geq 2 t^{-1/2}} \sup_{x\in\RR^2}\int_{A_2} K(t,z-x,y)|V(z)|dz
\leq
c \sup_{x\in\RR^2} \int_0^t\int_{\RR^2} g(s,x,z) |V(z)|dzds\,.
$$

It remains now to consider
$$
 \sup_{|y|\geq 2 t^{-1/2}} \sup_{x\in\RR^2}\int_{A_3} K(t,z-x,y)|V(z)|dz\,.
$$
Given $|y|\geq 2t^{-1/2}$ we let 
$$\mathcal{O}_y
=\left[\begin{array}{rr} y_1 |y|^{-1} & y_2 |y|^{-1}  \\ -y_2|y|^{-1} & y_1 |y|^{-1} \end{array} \right].
$$
Note that $\mathcal{O}_y
$ is a rotation matrix in $\RR^2$ such that $\mathcal{O}_y \,y=(|y|,0)$. 
Then, substituting 
$z$ by
$t^{1/2}\mathcal{O}_y^{-1} z$, we obtain
\begin{align}\label{eq:D_3}
\int_{A_3} K(t,z-x,y)|V(z)|dz
&= \int_{D_r} K(1,z,(r,0))\, U(z) dz\,,
\end{align}
where
$$
r=t^{1/2}|y|\,,
\quad
D_r=\{z\in\RR^2\colon z_1 \geq 0 \mbox{ and }  2 \leq |z|\leq r\}\,,
\quad
U(z)= t|V(t^{1/2}\mathcal{O}_y^{-1} z+x)|\,.
$$
Combining
\eqref{eq:D_3}
and
Lemma~\ref{lem:D} we get for $|y|\geq 2 t^{-1/2}$,
\begin{align*}
\int_{A_3} K(t,z-x,y)|V(z)|dz
&\leq 
c \sup_{w\in\RR^2}
\int_{|z|\leq 2} t|V(t^{1/2}\mathcal{O}_y^{-1}z+t^{1/2}\mathcal{O}_y^{-1}w+ x)|dz\\
&\leq
c \sup_{\widetilde{w}\in\RR^2}
\int_{|z|\leq \sqrt{4t}} |V(z+\widetilde{w})|dz\,.
\end{align*}
Thus by 
\eqref{est:A2},
$$
 \sup_{|y|\geq 2 t^{-1/2}} \sup_{x\in\RR^2}\int_{A_3} K(t,z-x,y)|V(z)|dz
\leq
 c \sup_{x\in\RR^2} \int_0^t\int_{\RR^2} g(s,x,z) |V(z)|dzds\,.
$$
This finally gives the desired estimate and ends the proof for $d=2$.

Now, let $d=1$.
Using \cite[Lemma~4.2]{MR3200161} with
$k(x)=\sqrt{t}\left(1+t|y|^2\right)^{-1/2}
\ind_{|x|\leq t|y|}$ and $K(x)=\sqrt{t}$
we get for $r>0$,
\begin{align*}
\|K(V,t)\|_{\infty}
&\leq \sup_{x,y\in\RR} \int_{\RR}
\sqrt{t}\left(1+t|y|^2\right)^{-1/2}
\ind_{|z-x|\leq t|y|} |V(z)|dz\\
&\leq 
\sup_{x,y\in\RR}
\left(1+\frac{\sqrt{4t}}{r}\left(\frac{t|y|^2}{1+t|y|^2} \right)^{1/2}\right)
 \int_{|z|<r} \sqrt{t} |V(z)|dz
\leq \left(1+\frac{\sqrt{4t}}{r}\right)  \sup_{x\in\RR}  \sqrt{t} \!\!\!\int_{|z|<r} |V(z)|dz\,.
\end{align*}
Eventually, we put $r=\sqrt{4t}$ and use
\eqref{est:A1}, which ends the proof.
\end{proof}

\section{Corollaries and proof of Theorem~\ref{thm:t1}}

We will now give corollaries of Theorems \ref{thm:dgeq4} -- \ref{thm:d2d1}. We will seperately consider the cases $d \ge 4$, $d=3$, $d=2$ and $d=1$.
We begin with $d\geq 4$ and an aftermath of Theorem~\ref{thm:dgeq4}.
\begin{cor}\label{cor:d4}
Let $d\geq 4$. There is compactly supported $V\leq 0$ such that
\begin{enumerate}[label*=(\roman*)]
\item
\eqref{est:gaus_b} holds with $\varepsilon_1< 1<\varepsilon_2$ arbitrarily close to $1$ ,
\item \eqref{est:gaus} holds,
\item  \eqref{est:sharp_time} does not hold 
.
\end{enumerate}
By considering $-V$ we can obtain a similar non-negative example.
\end{cor}
\begin{proof}
We take $V\leq 0$ from Theorem~\ref{thm:dgeq4}. We justify all statements by using
parts (a), (b), (c) and (d) of the theorem along with the references indicated below. Namely,
\begin{enumerate}[label*=$\vcenter{\hbox{\tiny$\bullet$}}$]
	\item $V$ is compactly supported by (a),
	\item \textit{(i)} follows from (b) and \cite[Theorem~1A]{MR1994762},
\item \textit{(ii)} follows from (c) and \cite[p. 556 and Corollary~A]{MR1772429},
\item  \textit{(iii)} follows from (d), Corollary~\ref{cor:reduction} and Lemma~\ref{lem:comb}. 
\end{enumerate}
For a non-negative example we may need to multiply $-V$ by a small constant to obtain (c$'$) $\|\Delta^{-1} V\|_{\infty}< \varepsilon$ (small) and use for instance \cite[Theorem~1.4]{MR3200161} to get {\it (ii)}.
\end{proof}

A similar argumentation based on Theorem~\ref{thm:d3}, \cite[Theorem~1A and~1B]{MR1994762}, Corollary~\ref{cor:reduction} and Lemma~\ref{lem:comb} gives consequences for $d=3$.
As pointed out after Theorem~\ref{thm:d3} we cannot expect an example of $V\leq 0$ that satisfies \eqref{est:gaus}, but not \eqref{est:sharp_time}.
 
\begin{cor}\label{cor:d3}
Let $d=3$. There is $V\leq 0$
(of unbounded support)
such that
\begin{enumerate}[label*=(\roman*)]
\item
\eqref{est:gaus_b} holds
with 
$\varepsilon_1< 1<\varepsilon_2$
arbitrarily close to $1$,
\item 
\eqref{est:sharp_time} fails to hold.
\end{enumerate}
\end{cor}

Here is what results from
Theorem~\ref{thm:d2d1} for $d=2$.

\begin{cor}\label{cor:d2_1}
Let $d=2$. We have
\begin{enumerate}
\item[1)]  $V\in \mathcal{K}_2$ if and only if $\lim_{T\to 0^+}\|S(V)\|_{T,\infty}=0$.
\item[2)] $V\in \widehat{\mathcal{K}}_2$  if and only if $\|S(V)\|_{T,\infty}<\infty$ for some (every) $T>0$.
\item[3)] If $V\leq 0$, then \eqref{est:sharp_time} holds if and only if
$V\in \widehat{\mathcal{K}}_2$.
\end{enumerate}
\end{cor}
\begin{proof}
The first two statements follow from
Theorem~\ref{thm:d2d1} and the definitions of 
$\mathcal{K}_2$ and  $\widehat{\mathcal{K}}_2$.
The last one follows from Lemma~\ref{lem:comb} and {\it 2)}.
\end{proof}

Finally we focus on $d=1$ in view of Theorem~\ref{thm:d2d1}.

\begin{cor}\label{cor:d1_1}
Let $d=1$. 
The following conditions are equivalent
\begin{enumerate}
\item[a)] $V\in \mathcal{K}_1$,
\item[b)] $V\in \widehat{\mathcal{K}}_1$,
\item[c)] $\sup_{x\in\Rd} \int_{|z-x|\leq 1} |V(z)|<\infty$,
\item[d)] $\lim_{T\to 0^+}\|S(V)\|_{T,\infty}=0$,
\item[e)] $\|S(V)\|_{T,\infty}<\infty$ for some (every) $T>0$. 
\end{enumerate}
\end{cor}
\begin{proof}
The equivalence of  {\it a)}, {\it b)} and {\it c)} is well known and follows for instance from
\eqref{est:A1}.
Part {\it a)} is equivalent to {\it d)}, and part {\it b)} to {\it e)}
by
Theorem~\ref{thm:d2d1}.
\end{proof}

\begin{cor}\label{cor:d1_2}
Let $d=1$. If $V$ is of fixed sign, then \eqref{est:sharp_time} holds if and only if
$V\in \mathcal{K}_1$.
\end{cor}
\begin{proof}
The equivalence follows from Lemma~\ref{lem:comb} and Corollary~\ref{cor:d1_1}.
\end{proof}

\begin{proof}[Proof of Theorem~\ref{thm:t1}]
We justify statements in Table~\ref{t:1}. 
We refer to 'local in time' and 'global in time' column as the 'first' and the 'second' column, respectively.
If $d\geq 4$, the lack of the equivalence in both columns is an aftermath of Corollary~\ref{cor:d4} (also since \eqref{est:sharp_uni} implies \eqref{est:sharp_time}).
If $d=3$, the negative answer in the 'first' column results from  Corollary~\ref{cor:d3}.
Before we move forward, we note that
for $V\leq 0$, by the Duhamel formula,
$$
\int_0^t\int_{\Rd} G(s,x,z)|V(z)|g(t-s,z,y)dzds\leq g(t,x,y)\,.
$$
Thus, by integrating in $x$ variable over $\Rd$, we see that \eqref{est:gaus} implies
\begin{align}\label{nes1}
\sup_{t>0,\,y\in\Rd} \int_0^t \int_{\Rd} |V(z)|\,g(t-s,z,y)dzds= \sup_{x\in\Rd} \int_0^{\infty} \int_{\Rd}g(s,x,z) |V(z)|dzds <\infty\,,
\end{align}
while \eqref{est:gaus_b} necessitates
\begin{align}\label{nes2}
\sup_{0<t\leq T,\,y\in\Rd} \int_0^t \int_{\Rd} |V(z)|\,g(t-s,z,y)dzds= \sup_{x\in\Rd} \int_0^T \int_{\Rd}g(s,x,z) |V(z)|dzds <\infty\,.
\end{align}
Therefore, the positive answer in dimension $d=3$ in the 'second' column  follows from \eqref{nes1} and
\cite[Corollary~1.5 and (8)]{MR3914946} (or see \cite[Page 6]{MR3914946}).
The remaining two positive answers in 'global in time' column (dimensions $d=2$, $d=1$) also follow from \eqref{nes1}, this time complemented with
Theorem~\ref{thm:d2d1} and \cite[Lemma~1.1]{MR3914946}.
The two positive answers in 'local in time' column (dimensions $d=2$, $d=1$) follow from 
\eqref{nes2}, Theorem~\ref{thm:d2d1} and the first statement of Lemma~\ref{lem:comb} (see also \cite[Lemma~1.1]{MR3914946}).
\end{proof}

\bibliographystyle{abbrv}
%plain}
%\bibliography{kvs-bib}

\end{document}